\documentclass[a4paper, 12pt]{amsart}

\usepackage{hyperref}
\usepackage{amssymb}

\theoremstyle{plain}
\newtheorem{theorem}{Theorem}[section]
\newtheorem{lemma}[theorem]{Lemma}

\newtheorem{corollary}[theorem]{Corollary}
\newtheorem{proposition}[theorem]{Proposition}

\theoremstyle{definition}

\newtheorem{definition}[theorem]{Definition}
\newtheorem{remark}[theorem]{Remark}
\newtheorem{notation}[theorem]{Notation}

\newtheorem*{todo*}{Todo}

\newcommand{\Hs}{\mathcal H}
\newcommand{\G}{\mathcal G}
\newcommand{\sN}{N}

\newcommand{\Iso}{\operatorname{Iso}}
\newcommand{\supp}{\operatorname{supp}}

\newcommand{\lsp}{\operatorname{span}}
\newcommand{\clsp}{\overline{\operatorname{span}}}

\newcommand{\dom}{\operatorname{dom}}
\newcommand{\cod}{\operatorname{cod}}

\newcommand{\units}[1]{#1^{\times}}

\newcommand{\Clc}[2]{C_{\operatorname{lc}}(#1, \units{#2} \cup \{0\})}

\numberwithin{equation}{section}

\title{Reconstruction of graded groupoids from  graded Steinberg algebras}

\author{Pere Ara}
\address{Department of Mathematics\\
Universitat Auto\'noma de Barcelona\\
 08193 Bellaterra (Barcelona), Spain}
\email{para@mat.uab.cat}
\author{Joan Bosa}
\address{
School of Mathematics and Statistics\\
University of Glasgow\\
15 University Gardens, G12 8QW\\ Glasgow, United Kingdom}
\email{joan.bosa@glasgow.ac.uk}

\author{Roozbeh Hazrat}
\address{
Centre for Research in Mathematics\\
Western Sydney University\\
Australia}
\email{r.hazrat@westernsydney.edu.au}

\author{Aidan Sims}
\address{School of Mathematics and Applied Statistics\\
University of Wollongong\\
NSW 2522, Australia}
\email{asims@uow.edu.au}

\date{\today}
\thanks{The first and second-named authors were partially supported by the grants DGI MICIIN
MTM2011-28992-C02-01 and MINECO MTM2014-53644-P. The second author is supported by the
Beatriu de Pin\' os fellowship (2014 BP-A 00123). This research was supported by Australian
Research Council grant DP150101598.}

\begin{document}

\begin{abstract}
We show how to reconstruct a graded ample Hausdorff groupoid with topologically principal
neutrally graded component from the ring structure of its graded Steinberg algebra over
any commutative integral domain with~1, together with the embedding of the canonical
abelian subring of functions supported on the unit space. We deduce that
diagonal-preserving ring isomorphism of Leavitt path algebras implies $C^*$-isomorphism
of $C^*$-algebras for graphs $E$ and $F$ in which every cycle has an exit.
\end{abstract}

\maketitle

\section{Introduction}

Since the independent introduction of Leavitt path algebras by Abrams--Aranda-Pino
\cite{AA-P} and by Ara--Moreno--Pardo \cite{AMP}, there has been a great deal of interest
in the many parallels between the theory of Leavitt path algebras and that of graph
$C^*$-algebras---particularly because the reasons for these parallels are frequently not
apparent from the standard generators-and-relations picture. A key unresolved conjecture,
due to Abrams and Tomforde, says that if $E$ and $F$ are graphs whose complex Leavitt
path algebras are isomorphic as rings, then they have isomorphic $C^*$-algebras (see
\cite[Question~6]{TomfordeList}). In this paper, we make some progress on this question
by studying diagonal-preserving isomorphisms of Steinberg algebras. We confirm a slight
weakening of Abrams and Tomforde's conjecture: if $E$ and $F$ are graphs in which every
cycle has an exit and there is a ring-isomorphism $L_{\mathbb{C}}(E) \cong
L_{\mathbb{C}}(F)$ that respects the canonical diagonal, then there is a
diagonal-preserving isomorphism $C^*(E) \cong C^*(F)$.

Our approach uses Steinberg algebras, introduced in \cite{Ste10} and, independently, in
\cite{CFST14}. The groupoid $C^*$-algebra of a groupoid $\G$ is a norm completion of the
convolution algebra of continuous, compactly supported functions from $\G$ to
$\mathbb{C}$. When $\G$ is totally disconnected, the Steinberg algebra $A_R(\G)$ for a
ring $R$ is the convolution algebra of locally constant functions from $\G$ to $R$; so
$A_{\mathbb{C}}(\G)$ is a dense subalgebra of $C^*(\G)$. In particular, when $\G$ is the
graph groupoid \cite{KPRR} of a directed graph $E$, so that $C^*(\G)$ is the graph
algebra $C^*(E)$, the Steinberg algebras $A_R(\G)$ are precisely the Leavitt path
$R$-algebras $L_R(E)$ (see \cite[Remark~4.4]{CFST14} and \cite[Example~3.2]{CS15}).

The algebra $C_0(\G^{(0)})$ of continuous complex-valued functions vanishing at infinity
on $\G^{(0)}$ is a commutative $C^*$-subalgebra of $C^*(\G)$, and the algebra $D$ of
locally constant functions from $\G^{(0)}$ to $R$ is a commutative subalgebra of each
$A_R(\G)$. In \cite{Ren2008} (see also \cite{Kumjian}), Renault showed that if $\G$ is
topologically principal, then it can be reconstructed from the data $(C^*(\G),
C_0(\G^{(0)}))$. So if $\G$ and $\Hs$ are topologically principal groupoids, then there
is an isomorphism $C^*(\G) \cong C^*(\Hs)$ that carries $C_0(\G^{(0)})$ to
$C_0(\Hs^{(0)})$ if and only if $\G \cong \Hs$. These results have been used recently to
prove remarkable results about continuous orbit equivalence rigidity for symbolic
dynamical systems \cite{BCW, Li, MM}.

In this paper, we prove that if $\G$ is ample and Hausdorff, $c : \G \to \Gamma$ is a
1-cocycle taking values in a discrete group, $c^{-1}(0)$ is a topologically principal
subgroupoid of $\G$, and $R$ is a commutative integral domain with~1, then $D$ is a
maximal commutative subring of the ring $A_R(\G)$, and we can recover $\G$ from the pair
$(A_R(\G), D)$ regarded as a $\Gamma$-graded ring with distinguished commutative subring.
As a direct consequence, we deduce the following. Suppose that $E$ and $F$ are directed
graphs in which every cycle has an exit, and suppose that there is a commutative integral
domain $R$ with~1 for which there is a ring isomorphism $\pi : L_R(E) \cong L_R(F)$ that
such $\pi(s_{\mu} s_{\mu^*})s_\eta s_{\eta^*} = s_\eta s_{\eta^*} \pi(s_{\mu} s_{\mu^*})$
for every path $\mu$ in $E$ and every path $\eta$ in $F$. Then there is a
diagonal-preserving isomorphism $C^*(E) \cong C^*(F)$. We also make a little progress on
the vexing question \cite[Question~1]{TomfordeList}: are the Leavitt path algebras
$L_{2,K}$ and $L_{2-,K}$ isomorphic for a field $K$? Our results imply that there is no
\emph{diagonal-preserving} ring-isomorphism between these two algebras.

A similar result about Leavitt path algebras, of which we became aware late in the
preparation of this work, was obtained recently by Brown--Clark--an Huef \cite{BCaH}.
Neither our result nor theirs is a direct generalisation of the other, though: their
theorem requires a $^*$-ring isomorphism and that $E$ be row-finite with no sinks (no
sources with their conventions), whereas ours requires only a ring isomorphism and does
not insist that $E$ should be row-finite or have no sinks; but our result requires that
every cycle in $E$ have an exit whereas theirs does not. But our result has many further
applications; for example, to Kumjian--Pask algebras of higher-rank graphs, to algebras
associated to Cantor minimal systems, to the algebras $L^{ab}(E, C)$ associated to
separated graphs by Ara and Exel in \cite{AE}, and to all groupoids arising from partial
actions of countable discrete groups on totally disconnected metrisable spaces
\cite{Exelbook}. Indeed, by \cite[Theorems 5.17~and~6.3]{Ste10}, our result applies to
all algebras associated to inverse semigroups that are weak semilattices.

\subsection*{Acknowledgement} We are very grateful to the anonymous referee, whose
helpful comments have significantly improved the exposition of the paper, and have also
suggested interesting lines of further enquiry.

\section{Preliminaries}\label{sec:prelims}

\subsection{Groupoids and Inverse Semigroups}
We give a very brief introduction to Hausdorff ample groupoids. For more detail, see
\cite{exel:bbms2007, Paterson}.

A groupoid $\G$ is a small category with inverses. We denote the set of identity
morphisms of $\G$ by $\G^{(0)}$, and call it the \emph{unit space} of $\G$. So $\G^{(0)}
= \{\gamma\gamma^{-1} : \gamma \in \G\}$. For $\gamma \in \G$, we write $r(\gamma) :=
\gamma\gamma^{-1}$ and $s(\gamma) := \gamma^{-1}\gamma$. So $r,s : \G \to \G^{(0)}$
satisfy $r(\gamma)\gamma = \gamma = \gamma s(\gamma)$ for all $\gamma \in \G$. A pair
$(\alpha,\beta) \in \G \times \G$ is then \emph{composable} if and only if $s(\alpha) =
r(\beta)$. We write $\G^{(2)}$ for the set of all composable pairs.

For $U,V\subseteq \G$, we write
\begin{equation}\label{eq:set product}
    UV = \{\alpha\beta \mid \alpha\in U,\,\beta\in V\text{ and }r(\beta)=s(\alpha)\}.
\end{equation}
Given units $u,v \in \G^{(0)}$, we write, as usual (e.g. \cite{Renault}), $\G_u$ for $s^{-1}(u)$ and $\G^v$ for $r^{-1}(v)$.
We then write $\G^v_u$ for $\G^v \cap \G_u$. The \emph{isotropy group} at the unit
$u\in\G^{(0)}$ is then the group $\G_u^u$. We say $u$ has trivial isotropy if $\G^u_u=\{u\}$.
The \emph{isotropy subgroupoid} of $\G$ is $\Iso(\G):=\bigcup_{u\in\G^0} \G^u_u$.

We say that $\G$ is a Hausdorff groupoid if it is endowed with a Hausdorff topology under
which the range, source and inverse maps are continuous, and the composition map is
continuous with respect to the subspace topology on $\G^{(2)} \subseteq \G \times \G$.
This implies, in particular, that both $\G^{(0)}$ and $\Iso(\G)$ are closed in $\G$. We
say $\G$ is \emph{\'etale} if $r$ and $s$ are both local homeomorphisms. In this case,
the topology on $\G$ has a basis of \emph{local bisections}\footnote{In much of the
groupoid literature, these are simply called \emph{bisections}, or sometimes
\emph{slices}.}: subsets $U\subseteq \G$ such that $r_{|_U}$ and ${s_{|_U}}$ are
homeomorphisms onto open subsets of $\G^{(0)}$. This guarantees that $\G^{(0)}$ is clopen
in $\G$. We say that $\G$ is \emph{ample} if it is \'etale and $\G^{(0)}$ has a basis of
compact open sets. If $\G$ is an ample Hausdorff groupoid, then $\G$ admits a basis of
compact open local bisections; since $\G$ is Hausdorff, these compact open sets are
clopen. The composition map in an ample Hausdorff groupoid is an open map. In this paper
we deal only with Hausdorff groupoids. We will work frequently with \emph{topologically
principal} groupoids in which the set of units with trivial isotropy is dense in
$\G^{(0)}$.

Recall (see e.g. \cite{exel:bbms2007, Law98, Paterson, Ste10} for more details) that an
\emph{inverse semigroup} is a semigroup $S$ such that for each $s\in S$ there exists a
unique element $s^* \in S$ satisfying $ss^*s=s$ and $s^*ss^*=s^*$. We denote by $E(S)$
the set of idempotents of $S$, which is automatically a commutative semigroup. If $S$ has
an element $0$ such that $0s = 0 = s0$ for all $s \in S$, then we call $S$ an
\emph{inverse semigroup with 0}. There is a natural order on $E(S)$ given by  $e\leq f$
if and only if $ef=e$, and this order extends to a partial order on $S$ given by $s\leq
t$ if $s=et$ for some idempotent $e$ (in which case we may always take $e = ss^*$). Given
an ample Hausdorff groupoid $\G$, the collection $S_\G$ of compact open local bisections
of $\G$ forms an inverse semigroup\footnote{Indeed, since $\G$ is Hausdorff, $S_\G$ is a
Boolean inverse meet-semigroup (see, for example, \cite{Weh}) with meet given by
intersection. But since the additional Boolean meet-semigroup structure comes from the
concrete realisation of $S_\G$ as a collection of subsets of $\G$, we will not introduce
the formal axiomatisation of a Boolean inverse meet-semigroup here, but work concretely
with the usual set operations when manipulating elements of $S_\G$.} with 0 under the
multiplication given by~\eqref{eq:set product}, with $U^* = U^{-1} = \{\gamma^{-1} :
\gamma \in U\}$. We then have $E(S_\G) := \{U \in S_\G : U \subseteq \G^{(0)}\}$, and the
product in $E(S_\G)$ agrees with the intersection operation on subsets of $\G^{(0)}$. The
zero element of $S_\G$ is the empty local bisection $\emptyset$.

If $S$ is an inverse semigroup with 0, and $\Gamma$ is a discrete group, then a
\emph{$\Gamma$-grading} of $S$ is a map $c : S \setminus\{0\} \to \Gamma$ such that
whenever $s,t \in S$ satisfy $st \not= 0$, we have $c(st) = c(s)c(t)$.

\subsection{Graded Steinberg algebras}\label{subset:GSAs}

Let $\G$ be a Hausdorff ample groupoid and let $R$ be a commutative ring with~1. We write
$A_R(\G)$ for the space of all locally constant functions $f : \G \to R$ with compact
support. This becomes an $R$-algebra under the convolution product
\[
f * g(\gamma) = \sum_{r(\alpha)
    = r(\gamma)} f(\alpha) g(\alpha^{-1}\gamma)
    = \sum_{\alpha\beta = \gamma} f(\alpha)g(\beta),
\]
and pointwise addition and $R$-action. For any involution on $R$ (possibly the trivial
one) the algebra $A_R(\G)$ becomes a $^*$-algebra with $f^*(\gamma) = f(\gamma^{-1})^*$.
For $f \in A_R(\G)$, we write $\supp(f)$ for the support $\{\gamma \in \G : f(\gamma)
\not= 0\}$.

For the majority of this paper, we will be interested in the situation where $R$ is in
fact an integral domain. We will write $\units{R}$ for the group of units of $R$.

Note that $A_R(\G)=\lsp_R\{1_U\mid U\in S_\G\}$. Specifically, given $f \in A_R(\G)$, the
sets $f^{-1}(r)$, indexed by $r \in R$, are compact open sets, so each admit a finite
cover by elements of $S_\G$. Since $U \setminus V \in S_\G$ whenever $U,V \in S_\G$, we
can find, for each $r$ such that $f^{-1}(r) \not= \emptyset$, a finite set $F_r \subseteq
S_\G$ of mutually disjoint compact open local bisections with $f^{-1}(r) =\bigsqcup_{U
\in F_r} U $. We then have $f = \sum_{f^{-1}(r) \not= \emptyset} \sum_{U \in F_r} r\cdot
1_U$. That is, every element of $A_R(\G)$ can be written as an $R$-linear combination of
finitely many indicator functions of mutually disjoint compact open local bisections. We
will use this fact frequently, and, in this context, it will be useful to recall from
\cite[Proposition 4.5]{Ste10} that for $U, V \in S_\G$,
\[
1_U*1_V=1_{UV}\qquad\text{ and }\qquad 1_U^* = 1_{U^{-1}}.
\]

Let $\Gamma$ be a discrete group, and $c$ a continuous homomorphism from $\G$ to $\Gamma$
(that is, $c:\G\to\Gamma$ is a continuous groupoid cocycle). By \cite[Lemma~3.1]{CS15}
there is a $\Gamma$-grading of $A_R(\G)$ such that
\[
A_R(\G)_g = \{f \in A_R(\G) : \supp(f) \subseteq c^{-1}(g)\}\quad\text{ for all $g \in \Gamma$}.
\]
We say that a local bisection $U$ is \emph{homogeneous} if $c(U)$ is a singleton. The
collection of all homogeneous compact open local bisections of $\G$ is a $\Gamma$-graded
$^*$-closed subsemigroup of $S_\G$, denoted by $S_{\G, \star}$, under the grading
$\bar{c}(U) = g$ if $U \subseteq c^{-1}(g)$. Each homogeneous piece $A_R(\G)_g$ of the
Steinberg algebra is then precisely the $R$-linear span of indicator functions of
elements of $\bar{c}^{-1}(g)$. As above, each element of $A_R(\G)_g$ can be written as an
$R$-linear combination of indicator functions of finitely many mutually disjoint such
homogeneous local bisections.

Since we can regard $\G^{(0)}$ as a subgroupoid of $\G$, it makes sense to talk about the
Steinberg algebra $A_R(\G^{(0)})$, which is just the commutative algebra of locally
constant compactly supported functions from $\G^{(0)}$ to $R$ under pointwise operations.
Since $\G^{(0)}$ is clopen, there is an embedding $\iota : A_R(\G^{(0)}) \to A_R(\G)$
such that $\iota(f)|_{\G^{(0)}} = f$ and $\iota(f)|_{\G \setminus \G^{(0)}} = 0$. We use
this embedding to regard $A_R(\G^{(0)})$ as a commutative subalgebra of $A_R(\G)$. This
$A_R(\G^{(0)})$ contains local units\label{pg:localunits} for $A_R(\G)$: it contains a
set $\mathcal{E}$ of mutually commuting idempotents such that for every finite subset $X$
of $A_R(\G)$, there exists $e \in \mathcal{E}$ such that $X \subseteq e A_R(\G) e$.
Specifically, $\mathcal{E} = \{1_K : K \subseteq \G^{(0)}\text{ is compact and open}\}$
has the desired property because if $f \in A_R(\G)$, then $K := s(\supp(f)) \cup
r(\supp(f)) \subseteq \G^{(0)}$ is compact and open, and $1_K * f = f = f * 1_K$.

To keep our notation uncluttered, we shall write $D_{\G}$, or just $D$ when the groupoid
is clear, for $A_R(\G^{(0)}) \subseteq A_R(\G)$ throughout this note. Observe that when
$R$ is an integral domain, the set $E(D)$ of idempotent elements of $D$ is $\{1_K : K
\subseteq \G^{(0)}\text{ is compact and open}\}$.

We will need the following general result in the proof of our main theorem to show that a
homomorphism $\phi : A_R(\G) \to A_R(\Hs)$ of Steinberg algebras that carries $D_\G$ into
the relative commutant of $D_\Hs$ in fact carries $D_\G$ onto $D_\Hs$.

\begin{lemma}\label{lem:masa}
Let $\G$ be a topologically principal ample Hausdorff groupoid and $R$ be a commutative
integral domain with~1. Then, $D_\G$ is a maximal abelian subring of $A_R(\G)$.
\end{lemma}
\begin{proof}
Certainly $D_\G$ is an abelian subring of $A_R(\G)$, so we just have to show that it is a
maximal abelian subring. For this purpose, suppose that $f \in A_R(\G) \setminus D_\G$;
we must find $a \in D_\G$ such that $fa \not= af$. Fix $\alpha \in \supp(f) \setminus
\G^{(0)}$. Because $\supp(f) \subseteq \G$ is open and $\G$ is topologically principal,
we may assume that the isotropy at $s(\alpha)$ is trivial; therefore, $s(\alpha) \not =
r(\alpha)$. So we can choose disjoint compact open neighbourhoods $V,W \subseteq
\G^{(0)}$ of $r(\alpha)$ and $s(\alpha)$, respectively. We now have $(1_V f 1_W)(\alpha)
= f(\alpha) \not= 0$ whereas $(1_V1_Wf)(\alpha) = 0$. In particular, $f 1_W \not= 1_W f$.
\end{proof}

\section{Reconstructing the groupoid}\label{sec:main}

In this section we consider a commutative integral domain $R$ with~1, and groupoids $\G$
endowed with a continuous cocycle $c$ whose kernel is topologically principal. We show
how to reconstruct $(\G, c)$ from the pair $(A_R(\G), D)$, regarded as a graded ring with
distinguished abelian subring. Our goal is the following result, which we prove at the
end of the section. Throughout $\Gamma$ is a group and $e$ is its neutral element.

\begin{theorem}\label{thm:main}
Let $\G$ and $\Hs$ be ample Hausdorff groupoids, let $R$ be a commutative integral domain
with~1, let $c : \G \to \Gamma$ and $d : \Hs \to \Gamma$ be gradings by a discrete group,
and suppose that $c^{-1}(e)$ and $d^{-1}(e)$ are topologically principal. Let $D_\G
\subseteq A_R(\G)$ and $D_\Hs \subseteq A_R(\Hs)$ be the abelian subalgebras consisting
of functions supported on $\G^{(0)}$ and $\Hs^{(0)}$. Then there is a graded isomorphism
$\rho : A_R(\G) \to A_R(\Hs)$ such that $\rho(D_\G) \subseteq D_\Hs$ if and only if there
is an isomorphism $\bar\rho : \G \to \Hs$ such that $d \circ \bar\rho = c$.
\end{theorem}

\begin{remark}
In Theorem~\ref{thm:main}, we are regarding $A_R(\G)$ as a ring. In particular, even if
$R = \mathbb{C}$, so that $A_R(\G)$ and $A_R(\Hs)$ have natural $^*$-algebra structures,
the existence of a diagonal-preserving graded ring isomorphism $A_R(\G) \to A_R(\Hs)$ is
sufficient to guarantee isomorphism of the groupoids $\G$ and $\Hs$. {\it En passant} we
observe that it then follows that $A_R(\G)$ and $A_R(\Hs)$ are in fact isomorphic as
$^*$-algebras.
\end{remark}

\subsection{The normaliser of \texorpdfstring{$D_\G$}{D\_G}}
The first step in proving Theorem~\ref{thm:main} is to define and study what we call the
normalisers of $D_{\G}$. As discussed earlier, we often write $D$ instead of $D_\G$ for
the subalgebra $A_R(\G^{(0)})$ of $A_R(\G)$ . Our approach is based on Kumjian's work
\cite{Kumjian} on $C^*$-diagonals, and Renault's later work on Cartan subalgebras of
$C^*$-algebras \cite{Ren2008}. It is also related to Brown, Clark and an Huef's treatment
of Leavitt path algebras \cite{BCaH} which, in turn, is based on Brownlowe, Carlsen and
Whittaker's work on graph $C^*$-algebras \cite{BCW}; but we must make adjustments for the
lack of a $^*$-algebra structure here and to exploit the presence of a grading.

Let $\G$ be an ample Hausdorff groupoid, and let $c : \G \to \Gamma$ be a grading of $\G$
by a discrete group. Note that we have $D \subseteq A_R(\G)_e$, the
trivially-graded homogeneous subalgebra of $A_R(\G)$ since $c(u) = c(u^2) = c(u)^2$ for
each $u \in \G^{(0)}$.

We shall define what we call the graded normalisers of $D$, and show that they comprise
an inverse semigroup. Later we will establish that an appropriate quotient of this
inverse semigroup acts on the Stone spectrum of $D$ by partial homeomorphisms, and prove
that $\G$ is isomorphic to the groupoid of germs for this action. We are grateful to the
referee for suggesting a number of improvements in this section.

\begin{definition}
Let $A$ be a ring. For $n \in A$, let $V(n) := \{m \in A : mnm = m\text{ and }nmn = n\}$,
the set of \emph{generalised inverses} for $n$ in $A$. Note that  $mn$ and $nm$ are
idempotents for all $m\in V(n)$. Let $D$ be a commutative subring of $A$ that has local
units for $A$ (as described on page~\pageref{pg:localunits}. We say that $n \in A$ is a
\emph{normaliser} of $D$ if there exists $m \in V(n)$ such that
\begin{itemize}
\item[(N)] $m D n \cup n D m \subseteq D$.
\end{itemize}

In the grading setting, if $A = \bigoplus_{g \in \Gamma} A_g$ is a grading of $A$ by a discrete group
$\Gamma$, we say that a normaliser $n$ of $D$ is \emph{homogeneous} if $n \in c^{-1}(g)$
for some $g \in \Gamma$.
\end{definition}

\begin{notation}
If $A$ is a ring and $D \subseteq A$ is a commutative subring containing local units for
$A$, then we write $\sN_A(D)$, or just $\sN(D)$ for the collection of all normalisers of
$D$ in $A$. If $A$ is $\Gamma$-graded, we write $\sN_\star(D)$ for the collection of
homogeneous normalisers of $D$, and for $g \in \Gamma$, we write $\sN_g(D)$ for
$\sN_\star(G) \cap A_g$; we call the elements of $\sN_g(D)$ the \emph{$g$-homogeneous}
normalisers of $D$.
\end{notation}


Here, we are interested in the situation where $A = A_R(\G)$ is the Steinberg algebra of
an ample Hausdorff groupoid, $D$ is the diagonal subalgebra of locally constant
$R$-valued functions on $\G^{(0)}$, and the grading is given by a cocycle $c : \G \to
\Gamma$. Observe that then $D \subseteq A_R(\G)_e$.

Our first proposition, which is the linchpin of the paper, characterises the homogeneous
normalisers of $D$ when $c^{-1}(e)$ is topologically principal. We will use this in
Corollary~\ref{cor:nice normalisers} to show that $N_\star(D)$ is an inverse semigroup.
As suggested by the referee, we could probably deduce this from the general theory of
inverse semigroups by showing that $N_\star(D)$ is a regular semigroup with commuting
idempotents; but it will be important later to have the explicit description of elements
of $\sN_\star(D)$ given in the following proposition. Recall that $\units{R}$ denotes the
group of units of a ring $R$.

\begin{proposition}\label{prp:standard form}
Let $\G$ be an ample Hausdorff groupoid, and let $c : \G \to \Gamma$ be a grading of $\G$
by a discrete group. Let $R$ be a commutative ring with~1.
\begin{enumerate}
\item\label{it:normaliser forward} Let $U \subseteq c^{-1}(g)$ be a compact open
    local bisection. Suppose that $F$ is a finite collection of mutually disjoint
    compact open subsets $V$ of $U$ such that $U = \bigsqcup_{V \in F} V$. Fix units
    $a_V \in \units{R}\text{ for } V \in F$ and let
    \[
    n = \sum_{V \in F} a_V 1_V \quad\text{ and }\quad
    m = \sum_{V \in F} a^{-1}_V 1_{V^{-1}}.
    \]
    Then $n \in \sN(D)$, with $m \in V(n)$ satisfying (N).
\item\label{it:normaliser backward} Suppose that $R$ is an integral domain and
    $c^{-1}(e)$ is topologically principal. Suppose that $n \in \sN_g(D)$, and that
    $m \in V(n)$ satisfies (N). Then there exist a compact open local bisection $U
    \subseteq c^{-1}(g)$, a decomposition $U = \bigsqcup_{V \in F} V$ of $U$ into
    finitely many mutually disjoint compact open subsets, and units $a_V \in
    \units{R}$ indexed by $V \in F$ such that
    \[
    n = \sum_{V \in F} a_V 1_V \quad\text{ and }\quad
    m = \sum_{V \in F} a^{-1}_V 1_{V^{-1}}.
    \]
    We have $nm = 1_{r(U)}$ and $mn = 1_{s(U)}$.
\end{enumerate}
\end{proposition}
\begin{proof}
(\ref{it:normaliser forward}) Since $D = \lsp_R\{1_K : K \subseteq \G^{(0)}\text{ is
compact open}\}$, to see that $m$ satisfies~(N) it suffices to show that each $n 1_K m$
and each $m 1_K n$ is contained in $D$. For a compact open $K \subseteq \G^{(0)}$,
\[
n 1_K m = \sum_{V,W \in F} a_V a^{-1}_W 1_V 1_K 1_{W^{-1}}
    = \sum_{V, W \in F} a_V a^{-1}_W 1_{V K} 1_{(WK)^{-1}}.
\]
Since $s(V) \cap s(W) = \emptyset$ for distinct $V,W \in F$, it follows that $1_{VK}
1_{(WK)^{-1}} = 0$ unless $V = W$. So
\[
n 1_K m = \sum_{V \in F} a_V a^{-1}_V 1_{V K (VK)^{-1}}
    = \sum_{V \in F} 1_{r(VK)}
    = 1_{r(UK)}
\]
since the sets $V$ are mutually disjoint and cover $U$. Similarly, $m1_K n = 1_{s(KU)}$.
This establishes~(N).

Now, we show that $ m\in V(n)$. Applying the two identities just derived with $K = s(U)
\cup r(U)$ gives $nm = 1_{r(U)}$ and $mn = 1_{s(U)}$. Hence  $1_V mn = nm 1_V = 1_V$ for
each $V\in F$. So
\[
nmn = \sum_{V \in F} a_V 1_V 1_{s(U)}
    = \sum_{V \in F} a_V 1_{V s(U)}
    = \sum_{V \in F} a_V 1_V = n,
\]
and similarly $mnm = m$.

(\ref{it:normaliser backward}) Let $U_n := \supp(n) \subseteq c^{-1}(g)$. Since $R$ is an
integral domain, $\supp(m)\supp(n) = \supp(mn) \subseteq \G^{(0)} \subseteq \G_e$, and so
$U_m := \supp(m)$ is a subset of $c^{-1}(g^{-1})$. The sets $U_n$ and $U_m$ are compact
open sets because $n,m \in A_R(\G)$. We will show that $U_n$ is a local bisection and
that $U_m = U_n^{-1}$. First observe that $mn$ is an idempotent in $D = A_R(\G^{(0)})$.
Since $R$ is an integral domain, it has no nontrivial idempotents, and since
multiplication in $D$ is pointwise, we conclude that $mn = 1_K$ for some compact open set
$K$. We claim that $K = s(U_n)$. For this, first suppose that $u \in K$. Then
\[
0 \not= mn(u)
    = \sum_{\alpha\beta = u} m(\alpha)n(\beta)
    = \sum_{s(\alpha) = u} m(\alpha^{-1})n(\alpha).
\]
So $n(\alpha) \not= 0$ for some $\alpha \in \G_u$, giving $u \in s(\supp(n)) = s(U_n)$.
Hence $K \subseteq s(U_n)$. For the reverse inclusion, suppose that $u \in s(U_n)$; that
is, there exists $\alpha \in U_n$ satisfying $u = s(\alpha)$. Then
\[
0 \not= n(\alpha) = nmn(\alpha) = \sum_{\beta\gamma = \alpha} n(\beta)mn(\gamma) = \sum_{\beta\gamma = \alpha} n(\beta)1_K(\gamma).
\]
Since $K \subseteq \G^{(0)}$, the rightmost sum in the preceding equation collapses to
$n(\alpha)1_K(u)$. Moreover, as it is nonzero, we conclude that $u \in K$, and then, $mn
= 1_{s(U_n)}$. The same argument applied to the homogeneous normaliser $m$ shows that $nm
= 1_{s(U_m)}$. Similarly, $nm = 1_{r(U_n)}$ and $mn = 1_{r(U_m)}$. Therefore, $s(U_n) =
r(U_m)$ and $s(U_m) = r(U_n)$.

Now, fixing $u \in s(U_n)$, we show that the set $U_n u U_m$ consists of a single element
of $\G^{(0)}$. Since
\[
0 \not= 1_{s(U_n)}(u) = mn(u) = \sum_{\alpha \beta = u} m(\alpha)n(\beta) = \sum_{s(\alpha) = u} m(\alpha^{-1})n(\alpha),
\]
there exists $\alpha_0 \in U_n u$ such that $\alpha_0^{-1} \in U_m$; we must show that
$U_n u U_m = \{r(\alpha_0)\}$. Since inversion in $\G$ is a homeomorphism, and since
$n,m$ are locally constant, we can choose an open local bisection $V^0_{\alpha_0}$ such
that $\alpha_0 \in V^0_{\alpha_0} \subseteq U_n$, such that $(V^0_{\alpha_0})^{-1}
\subseteq U_m$, and such that $n$ is constant on $V^0_{\alpha_0}$ and $m$ is constant on
$(V^0_{\alpha_0})^{-1}$. Let $W^0_{\alpha_0^{-1}} := (V^0_{\alpha_0})^{-1}$. The sets
$U_n u \setminus \{\alpha_0\}$ and $u U_m \setminus \{\alpha_0^{-1}\}$ are finite and
discrete because $r$ and $s$ are local homeomorphisms and $U_n$ and $U_m$ are compact.
For each $\alpha \in U_n u \setminus \{\alpha_0\}$, choose an open local bisection
$V^0_\alpha$ with $\alpha \in V^0_\alpha \subseteq U_n$ such that $n$ is constant on
$V^0_\alpha$; and for each $\beta \in u U_m \setminus \{\alpha_0^{-1}\}$ choose an open
local bisection $\beta \in W^0_\beta \subseteq U_m$ with $m$ constant on $W^0_\beta$. The
set $Y := \Big(\bigcap_{\alpha \in U_n} s(V^0_\alpha)\Big) \cap \Big(\bigcap_{\beta \in u
U_m} r(W^0_\beta)\Big)$ is open. For each $\alpha \in U_n u$, put $V_\alpha := V^0_\alpha
Y$, and for each $\beta \in u U_m$, let $W_\beta := Y W^0_\beta$. Then:
\begin{itemize}
\item $W_{\alpha_0^{-1}} = V_{\alpha_0}^{-1}$;
\item each $V_\alpha$ is an open local bisection containing $\alpha$, and each
    $W_\beta$ is an open local bisection containing $\beta$;
\item $n$ is constant on each $V_\alpha$ and $m$ is constant on each $W_\beta$; and
\item $s(V_\alpha) = Y = r(W_\beta)$ for all $\alpha,\beta$.
\end{itemize}

As $U_n \subseteq c^{-1}(g)$ and $U_m \subseteq c^{-1}(g^{-1})$, we have $V_\alpha
W_\beta \subseteq c^{-1}(e)$ for all $\alpha,\beta$. Because $Y$ is open, and since
$c^{-1}(e)$ is topologically principal, we can find $y \in Y$ such that $\G^y_y \cap
c^{-1}(e) = \{y\}$. It follows that if $\mu,\nu \in U_n y$ are distinct then $r(\mu)
\not= r(\nu)$ (because otherwise $\mu^{-1}\nu$ would belong to $\G^y_y \cap c^{-1}(e)
\setminus \{y\}$), and similarly, if $\mu,\nu \in yU_n$ are distinct then $s(\mu) \not=
s(\nu)$. So we can choose a compact open neighbourhood $V'_\eta$ of each $\eta \in U_n y$
and a compact open neighbourhood $W'_\zeta$ of each $\zeta \in y U_m$ such that the
$r(V'_\eta)$ are mutually disjoint and the $s(W'_\zeta)$ are mutually disjoint; and we
can assume that $V'_\eta \subseteq V_\alpha$ if $\eta \in V_\alpha$, and similarly for
the $W'_\zeta$. Let $X := \big(\bigcap_{\eta \in U_n y} s(V'_\eta)\big) \cap
\big(\bigcap_{\zeta \in yU_m} r(W'_\zeta)\big)$. Since $1_{r(V'_\eta)} 1_{r(V'_{\eta'})}
= 0$ for distinct $\eta, \eta' \in U_n y$, and $1_{s(W'_{\zeta'})}1_{s(W'_\zeta)} = 0$
for distinct $\zeta,\zeta' \in yU_m$, we have
\begin{equation}\label{eq:off-diagonal}
(1_{r(V'_\eta)} n 1_X m 1_{s(W'_\zeta)})(\eta\zeta) = n(\eta)m(\zeta)
\end{equation}
for each $\eta \in U_n y$ and $\zeta \in yU_m$. Now we suppose, that $U_n u U_m$ is not a
singleton, and derive a contradiction. Either there exists $\alpha_1 \in U_n u \setminus
\{\alpha_0\}$ or there is $\beta_1 \in u U_m \setminus \{\alpha_0^{-1}\}$. We consider
the former case; the latter is similar. Let $\eta_0$ and $\eta_1$ be the unique elements
of $V_{\alpha_0} y$ and $V_{\alpha_1} y$. Since $W_{\alpha_0^{-1}} = V_{\alpha_0}^{-1}$,
the unique element of $y W_{\alpha_0^{-1}}$ is $\zeta_0 := \eta_0^{-1}$. As $m,n$
satisfy~(N), we have $n 1_X m \in D$ and therefore $1_{r(V'_{\eta_1})} n 1_X m
1_{s(W'_{\zeta_0})} \in D$. But~\eqref{eq:off-diagonal} and that $n$ is constant on
$V_{\alpha_1}$ and $m$ is constant on $W_{\alpha_0^{-1}}$ gives
\[
\big(1_{r(V'_{\eta_1})} n 1_X m 1_{s(W'_{\zeta_0})}\big)(\eta_1\zeta_0) = n(\eta_1)m(\zeta_0) = n(\alpha_1)m(\alpha^{-1}_0).
\]
We have $n(\alpha_1) \not= 0$ and $m(\alpha_0^{-1}) \not= 0$ because $\alpha_1 \in U_n =
\supp(n)$ and $\alpha^{-1}_0 \in U_m = \supp(m)$; since $R$ is an integral domain, we
deduce that $\big(1_{r(V'_{\eta_1})} n 1_X m 1_{s(W'_{\zeta_0})}\big)(\eta_1\zeta_0)
\not= 0$. Since $\zeta_0 = \eta_0^{-1} \not= \eta_1^{-1}$, we have $\eta_1 \zeta_0 \not
\in \G^{(0)}$, which contradicts $1_{r(V'_{\eta_1})} n 1_X m 1_{s(W'_{\zeta_0})} \in D$.

We have now established that $U_n$ is a local bisection, and symmetry shows that $U_m$ is
a local bisection. We also showed at the beginning of the preceding paragraph that if $u
\in s(U_n)$ then there exists $\alpha \in U_n u$ such that $\alpha^{-1} \in U_m$. Since
we now also know that $U_n u$ is a singleton, we deduce that $U_n^{-1} \subseteq U_m$;
and symmetry implies that in fact $U_m = U_n^{-1}$. Since $n$ and $m$ are locally
constant, we can express $U_n = \bigsqcup_{V \in F} V$ where the $V$ are mutually
disjoint compact open sets such that $n$ is constant (and nonzero) on $V$ and $m$ is
constant on $V^{-1}$ for each $V \in F$; say $n \equiv r_V$ on $V$ and $m \equiv r'_V$ on
$V^{-1}$. So $n = \sum_{V \in F} r_V 1_V$ and $m = \sum_{V \in F} r'_V 1_{V^{-1}}$. We
just have to show that each $r_V r'_V = 1$. Since $U$ is a local bisection, the sets
$r(V)$ are mutually disjoint, and so for $w \in r(U)$ there is a unique $V$ such that $w
\in r(V)$. We then have
\[
nm(w) = \sum_{r(\alpha) = w} n(\alpha)m(\alpha^{-1})
    = r_V r'_{V^{-1}}.
\]
We have $w \in r(U_n)$ and we proved that $nm = 1_{r(U_n)}$. Hence $r_V r'_{V^{-1}} = 1$
as required.
\end{proof}

\begin{remark}\label{rmk:normalisers}
\begin{enumerate}
\item If $U \subseteq c^{-1}(g)$ is a compact open local bisection, then $1_U \in
    \sN_g(D)$, with $m = 1_{U^{-1}}\in V(1_U )$ satisfying (N). Hence $A_R(\G) =
    \lsp_R(\sN_\star(D))$.
\item Since $D$ has local units for $A_R(G)$ (see page~\pageref{pg:localunits}), if
    $n \in \sN(D)$, with $m \in V(n)$ satisfying (N), then $nm, mn \in D$.
\end{enumerate}
\end{remark}

\subsection{The inverse semigroup \texorpdfstring{$\sN_\star(D)$}{ of homogeneous normalisers}}

We now show that $\sN_\star(D)$ is an inverse semigroup carrying a $\Gamma$-grading. In
particular, we will show that a quotient $\sN_\star(D)/{\sim}$ of this inverse semigroup
acts on the Stone spectrum $\widehat{D}$, so that we can construct from it a
$\Gamma$-graded groupoid of germs $(\sN_\star(D)/{\sim}) \times_\varphi \widehat{D}$.

We write $\Clc{\G}{R}$ for the set $\{f \in A_R(\G) : f(\G) \subseteq \units{R} \cup
\{0\}\}$ of functions in $A_R(\G)$ whose nonzero values are units. We define $^* :
\Clc{\G}{R} \to \Clc{\G}{R}$ by
\begin{equation}\label{eq:star op}
f^*(\gamma)
    = \begin{cases}
        f(\gamma^{-1})^{-1} &\text{ if $f(\gamma^{-1}) \not= 0$}\\
        0 &\text{ otherwise.}
    \end{cases}
\end{equation}
It will be convenient to write $r^* := r^{-1}$ for $r \in \units{R}$ and $0^*: = 0$. Under
this notation, we have $f^*(\gamma) = f(\gamma^{-1})^*$ for all $f \in \Clc{\G}{R}$ and
$\gamma \in \G$.

We define
\[
\Clc{\G}{R}_\star := \{f \in \Clc{\G}{R} : f\text{ is homogeneous}\},
\]
and for $g \in \Gamma$, we write $\Clc{\G}{R}_g := \Clc{\G}{R} \cap A_R(\G)_g$.

We can reinterpret Proposition~\ref{prp:standard form} as follows.

\begin{corollary}\label{cor:nice normalisers}
Let $\G$ be an ample Hausdorff groupoid, and let $c : \G \to \Gamma$ be a grading of $\G$
by a discrete group. Suppose that $c^{-1}(e)$ is topologically principal. Let $R$ be a
commutative integral domain with~1. Then
\begin{equation}\label{eq:S def}
\sN_\star(D) =  \{f \in \Clc{\G}{R}_\star : \supp(f)\text{ is a local bisection}\},
\end{equation}
and the map $f \mapsto f^*$ is an involution on $\sN_\star(D)$ satisfying $f^*f =
1_{s(\supp(f))}$ and $ff^* = 1_{r(\supp(f))}$ for $f \in \sN_\star(D)$. Moreover,
$\sN_\star(D)$ is an inverse semigroup, and there is a $\Gamma$-grading $\tilde{c} :
\sN_\star(D) \setminus \{0\} \to \Gamma$ defined by $\tilde{c}(f) = g$ if and only if $f
\in \sN_g(D)$.
\end{corollary}
\begin{proof}
Proposition~\ref{prp:standard form} gives~(\ref{eq:S def}) holds. If $f \in
\sN_\star(D)$, then clearly $f^*\in \sN_\star(D)$ and $f^{**} = f$. Since $R$ is an
integral domain, for $f,g \in \sN_\star(D)$, and $\gamma \in \G$, we have $\gamma \in
\supp (fg)$ if and only if $\gamma \in \supp (f)  \supp (g)$; and in this case, there are
unique elements $\alpha \in \supp (f)$ and $\beta \in \supp (g)$ such that $\gamma =
\alpha \beta$. If $\gamma \in \supp (g^*) \supp (f^*)$, it follows that
\begin{align*}
(fg)^*(\gamma)
    = (fg)(\gamma^{-1})^{-1}
    &=  (f(\eta^{-1})g(\xi^{-1}))^{-1}\\
    &= g(\xi^{-1})^{-1}f(\eta^{-1})^{-1}
    =  g^*(\xi)f^*(\eta)= (g^* f^*) (\gamma ),
\end{align*}
where $\xi\in \supp (g^*)$ and $\eta \in \supp (f^*)$ are the unique elements such that
$\gamma = \xi \eta$. On the other hand, if $\gamma \notin \supp (g^*) \supp (f^*)$, then
$(fg)^* (\gamma ) = 0= (g^*)(f^*) (\gamma )$. So $^*$ is an involution on $\sN_\star(D)$.
(We did not need commutativity of $R$ for this computation.)

If $f \in \sN_\star(D)$ and $\gamma \in \G$, then the displayed equation above shows that
$(ff^*)(\gamma ) =0$ except when $\gamma \in \supp (f) \supp (f)^{-1}=r(\supp (f))$,
because $\supp (f)$ is a local bisection. Moreover, if $\gamma = \alpha \alpha^{-1} =
r(\alpha)$ for $\alpha \in \supp (f)$, then $(ff^*)(\gamma) = f(\alpha ) f(\alpha^{-1}
)^{-1} =1$. Hence $ff^*= 1_{r(\supp (f))}$. An identical calculation shows that $f^*f =
1_{s(\supp(f))}$. It follows immediately that $ff^*f = f$ and $f^*ff^* = f^*$, so that
$f^*\in V(f)$. For $K \subseteq \G^{(0)}$ compact open, we have $f 1_K = f|_{\supp(f)K}
\in \sN_\star(D)$, and likewise $1_K f \in \sN_\star(D)$; thus, $f 1_K f^* =
(f1_K)(1_Kf^*)$ and $f^* 1_K f = (1_K f)^* (1_K f)$ belong to $D$ by the reasoning just
applied to $ff^*$ and $f^*f$. Since $D$ is the $R$-linear span of the $1_K$, it follows
that $f D f^* \cup f^* D f \subseteq D$. So we have established that the pair $(f, f^*)$
satisfies~(N). This implies that the map $f \mapsto f^*$ is an involution on
$\sN_\star(D)$ satisfying the desired properties.

To see that $\sN_\star(D)$ is an $\Gamma$-graded inverse semigroup, fix $f_1, f_2 \in
\sN_\star(D)$. Each $\gamma \in \supp(f_1f_2)$ can be factorised uniquely as $\gamma =
\gamma_1 \gamma_2$ with $\gamma_1 \in \supp(f_1)$ and $\gamma_2 \in \supp(f_2)$. Since
$\units{R}$ is closed under multiplication,
\[
(f_1f_2)(\gamma) = f_1(\gamma_1)f_2(\gamma_2) \in \units{R} \cdot \units{R} \in \units{R} \cup \{0\}.
\]
So $\sN_\star(D)$ is closed under multiplication. The set $\sN_\star(D)$ is clearly
closed under $^*$, and $^*$ is an involution on $\sN_\star(D)$ such that $ff^*f = f$ and
$f^*ff^* = f^*$ by the preceding paragraph. The idempotents of $\sN_\star(D)$ are of the
form $1_K$, where $K$ is an open compact subset of $\G^{(0)}$. So they commute, and hence
$\sN_\star(D)$ is an inverse semigroup by \cite[Theorem~1.3]{Law98}.

Let $S_{\G, \star}$ be the $\Gamma$-graded inverse semigroup of homogeneous compact open local
bisections of $\G$, with grading $\bar{c}$ as established in Section~\ref{subset:GSAs}. Since $R$ is
an integral domain, if $f_1, f_2 \in \sN_\star(D)$, then $\supp(f_1f_2) =
\supp(f_1)\supp(f_2)$. Hence $f_1f_2 = 0$ if and only if $\supp(f_1)\supp(f_2) =
\emptyset$. Since $\tilde{c}(f) = \bar{c}(\supp(f))$ for all $f$, it follows that
$\tilde{c}$ is a $\Gamma$-grading.
\end{proof}

\subsection{From \texorpdfstring{$\sN_\star(D)$}{homogeneous normalisers} to the inverse semigroup of compact open local bisections}

The collection $S_{\G, \star}$ of homogeneous compact open local bisections (including
the empty local bisection) of $\G$ forms an inverse semigroup under composition and with
involution $U \mapsto U^{-1}$. We now pass from the inverse semigroup $\sN_\star(D)$
described in the preceding section to a quotient $\sN_\star(D)/{\sim}$, which we prove is
isomorphic to $S_{\G, \star}$.

We write $E(D)$ for the boolean ring of idempotent elements of $D$.

\begin{lemma}\label{lem:S smaller}
Let $\G$ be an ample Hausdorff groupoid, and let $c : \G \to \Gamma$ be a grading of $\G$
by a discrete group. Suppose that $c^{-1}(e)$ is topologically principal. Let $R$ be a
commutative integral domain with~1, and let $\tilde{c} : \sN_\star(D) \setminus \{0\} \to
\Gamma$ be the $\Gamma$-grading of Corollary~\ref{cor:nice normalisers}. Then there is an
equivalence relation $\sim$ on $\sN_\star(D)$ defined by $f \sim h$ if and only if all of
the following three conditions are satisfied:
\begin{enumerate}
\item $\tilde{c}(f) = \tilde{c}(h)$,
\item\label{it:sim2} $f^*pf = h^*ph$ for every $p \in E(D)$, and
\item\label{it:sim3} $fpf^* = hph^*$ for every $p \in E(D)$.
\end{enumerate}
Moreover, the set $\sN_\star(D)/{\sim}$ is an inverse semigroup under the operations
$[f][g] = [fg]$ and $[f]^* = [f^*]$. The map $q : \sN_\star(D) \to S_{\G, \star}$ given
by $q(f) = \supp(f)$ induces an isomorphism of inverse semigroups $\tilde{q} :
\sN_\star(D)/{\sim} \to S_{\G,\star}$.
\end{lemma}
\begin{proof}
It suffices to show that $q : \sN_\star(D) \to S_{\G,\star}$ is a surjective semigroup
homomorphism, and that $q(f) = q(h)$ if and only if $f \sim h$.

Since $R$ is an integral domain, we have $\supp(fh) = \supp(f)\supp(h)$ for all $f,h \in
\sN_\star(D)$, and we have $\supp(f^*) = \supp(f)^{-1}$ by definition of $^*$ on
$\sN_\star(D)$. So the map $q$ is a semigroup homomorphism. It is surjective because each
homogeneous compact open local bisection $V$ satisfies $V = q(1_V)$.

To see that $q(f) = q(h) \iff f \sim h$, first suppose that $q(f) = q(h)$. Then $\supp(f)
= \supp(h)$. Since $\tilde{c}(f) = g$ if and only if $\supp(f) \subseteq c^{-1}(g)$, it
follows that $\tilde{c}(f) = \tilde{c}(h)$. If $p \in D$ is an idempotent, then $p = 1_K$
for some compact open $K \subseteq \G^{(0)}$. We then have $pf = f|_{K\supp(f)}$ and $ph
= h|_{K\supp(h)} = h|_{K\supp(f)}$. Thus, Corollary~\ref{cor:nice normalisers} applied to
$pf, ph \in \sN_\star(D)$ shows that
\begin{align*}
f^*pf &= (pf)^*(pf) = 1_{s(K\supp(f))} = 1_{s(Kq(f))}\\
    &= 1_{s(Kq(h))} = 1_{s(K\supp(h))} = (ph)^*(ph) = h^*ph,
\end{align*}
and similarly $fpf^* = hph^*$.

Now suppose $f \sim h$. We must show that $\supp(f) = \supp(h)$. Since $\tilde{c}(f) =
\tilde{c}(h)$, we have $\supp(f) \supp(h)^{-1} \subseteq c^{-1}(e)$ which is
topologically principal. Putting $p = 1_{r(\supp(f)) \cup r(\supp(h))}$
in~(\ref{it:sim2}), we see that $f^*f = h^*h$, and so $s(\supp(f)) = s(\supp(h))$. We
will show that $\supp (f) \cap \supp (h)$ is dense in $\supp (f)$. Let $U$ be a non-empty
open subset of $\G$ contained in $\supp (f)$. Since $c^{-1}(e)$ is topologically
principal and $r(U)$ is open, there is $\alpha \in U$ such that the isotropy at
$r(\alpha)$ is trivial. Let $\beta \in \supp(h)$ be the unique element with $s(\alpha) =
s(\beta)$. We suppose that $\beta\not= \alpha$ to derive a contradiction. Since the
isotropy at $r(\alpha)$ is trivial, and $\beta \not=\alpha$, we have $r(\beta) \not=
r(\alpha)$; thus, there exist disjoint compact open sets $V,W \subseteq \G^{(0)}$ such
that $r(\alpha) \in V$ and $r(\beta) \in W$. Let $X = s(V \supp(f)) \cap s(W \supp(h))$.
Then $(f 1_X f^*)(r(\alpha)) = 1$ and $(h 1_X h^*)(r(\alpha)) = 0$,
contradicting~(\ref{it:sim3}). This shows that $\supp (f) \cap \supp (h)$ is dense in
$\supp (f)$. Since $\G$ is Hausdorff, the compact sets $\supp(f)$ and $\supp(h)$ are
closed, and so $\supp (f) \cap \supp (h)$ is closed. Hence $\supp (f) \cap \supp (h)=
\supp (f)$. Similarly $\supp (f) \cap \supp (h)= \supp (h)$ and thus $\supp (f) = \supp
(h)$.
\end{proof}

Since $R$ has no nontrivial idempotent elements, the Boolean ring $E(D)$ is precisely the
set $\{1_K : K \subseteq \G^{(0)}\text{ is compact open}\}$, and so corresponds to the
Boolean algebra of compact open subsets of $\G^{(0)}$. We write $\widehat{D}$ for the
Stone spectrum of $E(D)$: that is, the space of Boolean-ring homomorphisms $\pi : E(D)
\to \{0,1\}$. By Stone duality, there is a homeomorphism $\varepsilon : \G^{(0)} \to
\widehat{D}$ such that $\varepsilon_u(p) := p(u)$ for $p \in E(D)$; the inverse of this
map takes a Boolean-ring homomorphism $\pi : E(D) \to \{0,1\}$ to the unique point in
$\big(\bigcap_{\pi(1_K) = 1} K\big) \setminus \big(\bigcup_{\pi(1_K) = 0} K\big)$.

Recall that there is an action $\theta$ of $S_{\G, \star}$ on $\G^{(0)}$ such that
$\dom(\theta_V) = s(V)$, $\cod(\theta_V) = r(V)$ and $\theta_V(s(\alpha)) = r(\alpha)$
for all $\alpha \in V$.

\begin{lemma}\label{lem:S action}
Let $\G$ be an ample Hausdorff groupoid, let $c : \G \to \Gamma$ be a grading of $\G$ by
a discrete group with $c^{-1}(e)$ topologically principal, let $R$ be a commutative
integral domain with~1, and let $\sN_\star(D)/{\sim}$ be the inverse semigroup of
Lemma~\ref{lem:S smaller}. With $\varepsilon : \G^{(0)} \to \widehat{D}$ and $\theta :
S_{\G, \star} \curvearrowright \G^{(0)}$ as above, and $\tilde{q} : \sN_\star(D)/{\sim}
\to S_{\G,\star}$ as in Lemma~\ref{lem:S smaller}, we have
\begin{equation}\label{eq:theta-varphi}
\varepsilon_{\theta_{\tilde{q}(f)}(u)}(p)
    = \varepsilon_u(f^* p f)\quad\text{ for all $f \in \sN_\star(D)/{\sim}$, $u \in s(\supp(f))$ and $p \in E(D)$.}
\end{equation}
\end{lemma}
\begin{proof}
Fix $f \in \sN_\star(D)$, $u \in s(\supp(f))$ and $p \in E(D)$. Then $p = 1_K$ for some
compact open $K \subseteq \G^{(0)}$. Let $\alpha \in \supp(f)$ be the unique element with
$s(\alpha) = u$. Then
\[
\varepsilon_{\theta_{\tilde{q}(f)}(u)}(p)
    = p(\theta_{\tilde{q}(f)}(u))
    = 1_K(r(\alpha)).
\]
Also,
\begin{align*}
\varepsilon_u(f^* p f)
    = (f^* 1_K f)(u)
    &= \sum_{\beta\gamma = u} (1_Kf)^*(\beta) f(\gamma)
    = (1_K f)^*(\alpha^{-1})f(\alpha)\\
    &= (1_K f)(\alpha)^*f(\alpha)
    = 1_K(r(\alpha)) f(\alpha)^*f(\alpha)
    = 1_K(r(\alpha)).\qedhere
\end{align*}
\end{proof}

\subsection{Groupoids of germs}

If $\varphi$ is an action of a countable inverse semigroup $S$ on a locally compact
Hausdorff space $X$, then the groupoid of germs $S \times_\varphi X$ is defined as
follows (see for example \cite{exel:bbms2007, Paterson}). Define a relation $\sim$ on
$\{(s, x) \in S \times X : x \in \dom(\varphi_s)\}$ by $(s, x) \sim (s', y)$ if $x = y$
and there is an idempotent $p \in E(S)$ such that $x \in \dom(p)$ and $sp = s'p$. This is
an equivalence relation, and the collection $S \times_\varphi X$ of equivalence classes
for this relation is a locally compact \'etale groupoid with unit space $X$ and structure
maps $r([s,x]) = \varphi_s(x)$, $s([s,x]) = x$, $[s,\varphi_t(x)][t, x] = [st, x]$, and
$[s,x]^{-1} = [s^*, \varphi_s(x)]$. Moreover, if $\tilde{c} : S\setminus\{0\} \to \Gamma$
is a grading, then, as idempotent elements $p \in S \setminus\{0\}$ satisfy $\tilde{c}(p)
= e$, there is a grading of $S \times_\varphi X$ given by $[s,x] \mapsto \tilde{c}(s)$.

Proposition~5.4 of \cite{exel:bbms2007} implies that for graded ample Hausdorff groupoids
$\G$, the groupoid of germs for the action $\theta$ of $S_{\G, \star}$ on $\G^{(0)}$ is
canonically isomorphic to $\G$. Combining the preceding subsections with this result, we
recover $\G$ from $\sN_\star(D)$ and $D$.

\begin{lemma}\label{lem:action matchup}
Let $\G$ be an ample Hausdorff groupoid, let $c : \G \to \Gamma$ be a grading of $\G$ by
a discrete group with $c^{-1}(e)$ topologically principal, let $R$ be a commutative
integral domain with~1, and let $\sN_\star(D)/{\sim}$ be the inverse semigroup of
Lemma~\ref{lem:S smaller}. Then there is an action $\varphi$ of $\sN_\star(D)/{\sim}$ on
$\widehat{D}$ such that $\dom([f]) = \{\pi : \pi(1_{s(\supp(f))}) = 1\}$, $\cod([f]) =
\{\pi : \pi(1_{r(\supp(f))}) = 1\}$, and
\begin{equation}\label{eq:varphi description}
\varphi_{[f]}(\pi)(p) = \pi(f^* p f)\quad\text{ for all $f \in \sN_\star(D)/{\sim}$,
                $\pi \in \dom([f])$ and $p \in E(D)$.}
\end{equation}
Moreover, if $\tilde{q} : \sN_\star(D)/{\sim} \to S_{\G, \star}$ is the isomorphism of
Lemma~\ref{lem:S smaller}, then the homeomorphism $\varepsilon : \G^{(0)} \to
\widehat{D}$ intertwines $\theta_{\tilde{q}([f])}$ and $\varphi_{[f]}$ for every $f \in
S$.
\end{lemma}
\begin{proof}
Lemma~\ref{lem:S action} shows that the formula given for $\varphi$ satisfies
\[
\varphi_{[f]}(\varepsilon_u) = \varepsilon_{\theta_{\tilde{q}([f])}}
\]
for all $f$ and $u$, so the result follows by pulling the action $\theta$ back to an
action of $\sN_\star(D)/{\sim}$ via the isomorphism $\tilde{q}$.
\end{proof}

We now obtain our key result.

\begin{corollary}\label{cor:groupoid iso}
Let $\G$ be an ample Hausdorff groupoid, $c : \G \to \Gamma$ be a grading of $\G$ by a
discrete group with $c^{-1}(e)$ topologically principal, $R$ be a commutative integral
domain with~1, and $\sN_\star(D)/{\sim}$ be the inverse semigroup of Lemma~\ref{lem:S
smaller}. Let $\varphi : \sN_\star(D)/{\sim} \curvearrowright \widehat{D}$ be the action
of Lemma~\ref{lem:action matchup}. Then, there is an isomorphism $\pi :
(\sN_\star(D)/{\sim}) \times_\varphi \widehat{D} \to \G$ such that
\[
\pi\big(\big[[f], \varepsilon_{s(\alpha)}\big]) = \alpha
\]
for all $f \in \sN_\star(D)$ and $\alpha \in \supp(f)$. Moreover, $\pi$ intertwines the
grading on $(\sN_\star(D)/{\sim}) \times_\varphi \widehat{D}$ induced by $\tilde{c}$ and
the grading $c$ of $\G$.
\end{corollary}
\begin{proof}
Proposition~5.4 of \cite{exel:bbms2007} implies that there is an isomorphism
\[
\pi_0 : S_{\G, \star} \times_\theta \G^{(0)} \to \G
\]
such that $\pi_0([V, s(\alpha)]) = \alpha$ for all $V \in S_{\G, \star}$ and $\alpha \in
V$. This $\pi_0$ clearly carries the grading of $S_{\G, \star} \times_\theta \G^{(0)}$
induced by $c$ to $c$. The final statement of Lemma~\ref{lem:action matchup} shows that
$\tilde{q}$ and $\varepsilon$ induce an isomorphism of the action of
$\sN_\star(D)/{\sim}$ on $\widehat{D}$ with that of $S_{\G, \star}$ on $\G^{(0)}$. Thus,
there is an isomorphism $\pi' : (\sN_\star(D)/{\sim}) \times_\varphi \widehat{D} \to
S_{\G, \star} \times_\theta \G^{(0)}$ such that
$\pi'\big(\big[[f],\varepsilon_u\big]\big) = [\tilde{q}(f), u]$ for all $f$ and $u$, and
$\pi'$ takes the grading of $(\sN_\star(D)/{\sim}) \times_\varphi \widehat{D}$ induced by
$\tilde{c}$ to the grading of $S_{\G, \star} \times_\theta \G^{(0)}$ induced by $c$.
Therefore, $\pi = \pi_0 \circ \pi'$ is the desired isomorphism.
\end{proof}

\begin{remark}
As an alternative to Corollary~\ref{cor:groupoid iso}, one might aim to employ
Lawson--Lenz's noncommutative generalisation of Stone duality \cite[Theorem~3.25]{LL}
(see also \cite{Lawson:JAMS2010}) to recover $\G$ from $\sN_\star(D)/{\sim}$. Indeed, as
the anonymous referee points out, the results of this section could possibly be
reformulated as a statement of graded duality, between appropriate categories of graded
boolean inverse semigroups and graded \'etale groupoids. However, we will not pursue this
interpretation further here.
\end{remark}

\subsection{Proof of Theorem~\ref{thm:main}}

We can now prove our main theorem.

\begin{proof}[Proof of Theorem~\ref{thm:main}]
Clearly if $\bar\rho : \G \to \Hs$ is a graded isomorphism of groupoids, then there is a
graded isomorphism $\rho : A_R(\G) \to A_R(\Hs)$ given by $\rho(f) = f \circ
\bar\rho^{-1}$. Moreover, this isomorphism carries $D_\G$ to $D_\Hs$ because $\bar\rho$ carries
$\G^{(0)}$ to $\Hs^{(0)}$.

For the reverse implication, suppose that $\rho : A_R(\G) \to A_R(\Hs)$ is a graded ring
isomorphism with $\rho(D_\G) \subseteq D_\Hs$. Lemma~\ref{lem:masa} applied to
$c^{-1}(e)$ implies that $D_\G$ is maximal abelian subring of $A_R(\G)_e$, which implies
that $\rho(D_\G)$ is a maximal abelian subring of $A_R(\Hs)_e$  because $\rho$ is a
graded ring isomorphism. Since $\rho(D_\G)$ is contained in the abelian subring $D_\Hs$
of $A_R(\Hs)_e$, we deduce that $\rho(D_\G) = D_\Hs$. Thus, $\rho$ restricts to an
isomorphism of boolean rings $E(D_\G) \cong E(D_\Hs)$ inducing a homeomorphism $\rho^* :
\widehat{D}_\Hs \cong \widehat{D}_\G$ given by $\rho^*(\pi) = \pi \circ \rho$.

If $n$ is a homogeneous graded normaliser of $D_\G$ in $A_R(\G)$, then $\rho(n)$ is a
homogeneous normaliser of $D_\Hs$ in $A_R(\Hs)$: the conditions defining a normaliser
involve only the ring structure. Corollary~\ref{cor:nice normalisers} shows that
$\sN_\star(D_\G)$ is an inverse semigroup, as is $\sN_\star(D_\Hs)$. Thus, $\rho$
restricts to a graded isomorphism $\sN_\star(D_\G) \cong \sN_\star(D_\Hs)$ of inverse
semigroups. Since $\rho$ carries $E(D_\G)$ to $E(D_\Hs)$, the equivalence relations
$\sim$ on $\sN_\star(D_\G)$ and $\sN_\star(D_\Hs)$ defined in Lemma~\ref{lem:S smaller}
satisfy $f \sim h$ if and only if $\rho(f) \sim \rho(h)$. So $\rho$ induces a graded
isomorphism $\tilde\rho : \sN_\star(D_\G)/{\sim} \to \sN_\star(D_\Hs)/{\sim}$. The
definition of $\rho^*$ in the preceding paragraph and the formula~\eqref{eq:varphi
description} in Lemma~\ref{lem:action matchup} show that
$\rho^*(\varphi_{\tilde\rho([f])}(\pi)) = \varphi_{[f]}(\rho^*(\pi))$ for all $f \in
\sN_\star(D_\G)$ and $\pi \in \widehat{D}_\Hs$. So there exists a graded isomorphism
\[
\hat\rho : (\sN_\star(D_\G)/{\sim}) \times_\varphi \widehat{D}_\G \cong (\sN_\star(D_\Hs)/{\sim}) \times_\varphi \widehat{D}_\Hs
\]
that carries $\big[[f], \rho^*(\pi)\big]$ to $\big[[\rho(f)], \pi\big]$ for $f \in
\sN_\star(D_\G)$ and $\pi \in \widehat{D}_\Hs$. Now the graded isomorphisms $\pi_\G :
(\sN_\star(D_\G)/{\sim}) \times_\varphi \widehat{D}_\G \to \G$ and $\pi_\Hs :
(\sN_\star(D_\Hs)/{\sim}) \times_\varphi \widehat{D}_\Hs \to \Hs$ yield a graded
isomorphism $\bar\rho := \pi_\Hs \circ \hat\rho \circ \pi^{-1}_\G : \G \to \Hs$.
\end{proof}

\begin{remark}
It is worth discussing the extent to which the hypotheses on our main theorem are
necessary.
\begin{itemize}
\item If $c^{-1}(\G)_e$ is not topologically principal, then $D_\G$ is not
    necessarily maximal abelian in $A_R(\G)_e$; and, in addition, key steps in our
    analysis of the normalisers of $D_\G$ and of the quotient
    $\sN_\star(D_\G)/{\sim}$ break down. It is not clear, however, that this
    hypothesis is necessary to our main result: Brown--Clark--an Huef \cite{BCaH}
    show that the conclusion of our main theorem is valid for arbitrary graph
    groupoids with the trivial grading. Our theorem also holds in the special case
    that $\G$ has one unit (is a discrete group), $c$ is the trivial grading and $R =
    \mathbb{Z}$: it then reduces to the classical fact that the group-ring
    construction and the group-of-units construction are adjoint functors.
\item It is unclear whether it is necessary to assume that $R$ is an integral domain
    or that it is unital. These hypotheses are used heavily in our analysis, but we
    do not have a counterexample to our main result in their absence.
\item It is, however, necessary to make some assumptions on $R$: Let $R :=
    C_{\operatorname{lc}}(K)$, the ring of locally-constant complex-valued functions
    on the Cantor set. Since $K \cong K \sqcup K$, we have $R \cong R \oplus R$. Hence,
    for any ample Hausdorff groupoid $\G$, there exists a diagonal-preserving
    isomorphism $A_R(\G) \cong A_R(\G) \oplus A_R(\G) = A_R(\G \sqcup \G)$, whereas
    $\G$ and $\G \sqcup \G$ are not usually isomorphic.
\end{itemize}
\end{remark}

\section{Applications}\label{sec:apps}

\subsection{Topologically principal groupoids}

Here we record what our results say for ungraded ample groupoids $\G$. Given any groupoid
$\G$ we can endow it with the trivial grading $c : \G \to \{e\}$, and then apply our
main theorem. In this instance, we have $\sN_\star(D) = \sN(D)$ and $\Clc{\G}{R}_\star =
\Clc{\G}{R}$.

\begin{theorem}\label{Thm:Groupoid}
Let $\G$ be an ample Hausdorff groupoid and suppose that $\G$ is topologically principal.
Let $R$ be a commutative integral domain with~1. Then
\begin{enumerate}
\item $\sN(D) = \{n \in A_R(\G) : \supp(n)\text{ is a local bisection and } n(\G)
    \subseteq \units{R} \cup \{0\}\}$.
\item  For each $n \in \sN(D)$, there is a unique $n^* \in A_R(\G)$ such that $nDn^*
    \cup n^*Dn \subseteq D$, $nn^*n = n$ and $n^*nn^* = n^*$. Moreover, the element
    $n^*$ belongs to $\sN(D)$, and $\sN(D)$ is an inverse semigroup under $^*$.
\item There is an equivalence relation on $\sN(D)$ given by $n_1 \sim n_2$ if and
    only if $n^*_1 p n_1 = n^*_2 p n_2$ and $n_1 p n^*_1 = n_2 p n^*_2$ for every
    idempotent $p \in D$. Furthermore, $\sN(D)/{\sim}$ is an inverse semigroup under
    the operations inherited from $\sN(D)$.
\item There is an action $\varphi$ of $\sN(D)/{\sim}$ on the Stone spectrum
    $\widehat{D}$ of the collection of idempotent elements of $D$ such that
    $\varphi_{[n]}(\pi)(p) = \pi(n^* p n)$ for all $n \in \sN(D)$, $\pi \in
    \widehat{D}$ and every idempotent $p \in D$.
\item There is an isomorphism $\G \cong (\sN(D)/{\sim}) \times_\varphi \widehat{D}$
    that carries $\alpha \in \G$ to $\big[[1_V], \varepsilon_{s(\alpha)}\big]$ for
    any compact open local bisection $V$ containing $\alpha$.
\end{enumerate}
\end{theorem}
\begin{proof}
Let $c : \G \to \{e\}$ be the trivial grading. Corollary~\ref{cor:nice normalisers}
applied to this grading $c$ gives (1)~and~(2). Part~(3) follows from
Lemma~\ref{lem:S smaller} with the trivial grading, and~(4) follows from Lemma~\ref{lem:action
matchup}. Finally, the inverse of the isomorphism $\pi$ obtained from
Corollary~\ref{cor:groupoid iso} satisfies the formula described in~(5).
\end{proof}

\begin{remark}
The action of $\sN(D)$ on $\widehat{D}$ given by $(n \cdot \pi)(e) = \pi(n^* e n)$, which
induces an action of $\sN(D)/{\sim}$ in Theorem~\ref{Thm:Groupoid}(4) is usually
called the \emph{spectral action}, and it enjoys a universal property (see \cite{Ste10}
for more details). In particular, as described in \cite[Example 5.9]{Ste10}, this action
is the dual of the (right) Munn representation.

Moreover, the isomorphism described in Theorem~\ref{Thm:Groupoid}(4) induces the
isomorphism of Steinberg algebras $A_R(N(D)/{\sim}  \times_\varphi \hat{D})\cong A_R(\G)$
described in \cite[Theorem 6.3]{Ste10}; thus, Theorem~\ref{Thm:Groupoid} is, in a sense,
dual to Steinberg's result.
\end{remark}

\begin{corollary}\label{cor:ungraded gpds}
Let $\G$ and $\Hs$ be ample Hausdorff topologically principal groupoids and let $R$ be a
commutative integral domain with~1. The following are equivalent:
 \begin{enumerate}
  \item The groupoids $\G$ and $\Hs$ are isomorphic as topological groupoids.
  \item There is ring isomorphism $\rho:A_R(\G)\to A_R(\Hs)$ such that $\rho(D_\G)
      \subseteq D_\Hs$.
 \end{enumerate}
\end{corollary}
\begin{proof}
Apply Theorem~\ref{thm:main} to the trivial grading $\Gamma : \G \to \{e\}$.
\end{proof}

\subsection{Ring-isomorphisms of Leavitt path algebras}

In this short section, we make an observation about the implications of our results for
Leavitt path algebras. For background on Leavitt path algebras and on Abrams and
Tomforde's isomorphism conjecture, see \cite{AA-P, AT11, AMP, Tomforde:JA07,
Tomforde:JPAA11}.

Abrams and Tomforde conjectured that if $E$ and $F$ are graphs for which there is a ring
isomorphism $L_R(E) \cong L_R(F)$ for some ring $R$, then $C^*(E) \cong C^*(F)$ as
$C^*$-algebras. This conjecture remains open, but we make some headway (see also
\cite{BCaH} and Remark~\ref{rmk:BCaH} below).

\begin{corollary}\label{cor:graphs}
Suppose that $E$ and $F$ are graphs in which every cycle has an exit. Then the following
are equivalent:
\begin{enumerate}
\item\label{it:cor:graphs1} There exists a commutative integral domain $R$ with~1 for
    which there is an isomorphism $\pi : L_R(E) \to L_R(F)$ satisfying $\pi(s_{\mu}
    s_{\mu^*}) \in \lsp_R\{s_\eta s_{\eta^*} : \eta \in F^*\}$ for all $\mu \in E^*$.
\item\label{it:cor:graphs2} There exists a commutative integral domain $R$ with~1 for
    which there is an isomorphism $\pi : L_R(E) \to L_R(F)$ satisfying $\pi(s_{\mu}
    s_{\mu^*})s_\eta s_{\eta^*} = s_\eta s_{\eta^*} \pi(s_\mu s_{\mu^*})$ for all
    $\mu \in E^*$ and $\eta \in F^*$.
\item\label{it:cor:graphs3} For every $^*$-ring $S$ there exists a $^*$-isomorphism
    of $L_S(E)$ onto $L_S(F)$ that carries $\lsp_S\{s_\mu s_{\mu^*} : \mu \in E^*\}$
    to $\lsp_S\{s_\eta s_{\eta^*} : \eta \in F^*\}$.
\item\label{it:cor:graphs4} There is an isomorphism $\psi : C^*(E) \to C^*(F)$ such
    that $\psi(\clsp\{s_\mu s^*_\mu : \mu \in E^*\}) = \clsp\{s_\eta s^*_\eta : \eta
    \in F^*\}$.
\item\label{it:cor:graphs5} There is an isomorphism $\psi : C^*(E) \to C^*(F)$ such
    that $\psi(s_\mu s^*_\mu) s_\eta s^*_\eta = s_\eta s^*_\eta \psi(s_\mu s^*_\mu)$
    for all $\mu \in E^*$ and $\eta \in F^*$.
\end{enumerate}
\end{corollary}

Recall from \cite[Example~3.2]{CS15} that if $E$ is a directed graph, then there is an
isomorphism
\[
\alpha_E : L_R(E) \cong A_R(\G_E)
\]
that carries $\lsp_S\{s_\mu s_{\mu^*} : \mu \in E^*\}$ to $D_{\G_E}$. We will use this
isomorphism at a number of points in the proof of Corollary~\ref{cor:graphs}.

\begin{proof}[Proof of Corollary~\ref{cor:graphs}]
It is well known (see \cite{KPRR}) that the groupoid $\G_E$ of a directed graph $E$ is
topologically principal provided that every cycle in $E$ has an exit. So $\G_E$ and
$\G_F$ are topologically principal.

We first prove (\ref{it:cor:graphs1})\;$\iff$\;(\ref{it:cor:graphs2}). The implication
(\ref{it:cor:graphs1})~implies~(\ref{it:cor:graphs2}) is trivial. For the reverse,
observe that if $\pi$ is as in~(\ref{it:cor:graphs2}), then each $\pi(s_\mu s_{\mu^*})$
commutes with every element of $\lsp_R\{s_\eta s_{\eta^*} : \eta \in F^*\}$. Since the
latter is a maximal abelian subring by Lemma~\ref{lem:masa}, it follows that each
$\pi(s_\mu s_{\mu^*}) \in \lsp_R\{s_\eta s_{\eta^*} : \eta \in F^*\}$.

Next we prove that~(\ref{it:cor:graphs1}) implies
(\ref{it:cor:graphs3})~and~(\ref{it:cor:graphs5}). Suppose that~(\ref{it:cor:graphs1})
holds. Corollary~\ref{cor:ungraded gpds} implies that the graph groupoids $\G_E$ and
$\G_F$ are isomorphic; say $\rho : \G_F \to \G_E$ is an isomorphism. Then $\rho$
restricts to a homeomorphism $\G^{(0)}_F \to \G^{(0)}_E$. For each $*$-ring $S$, $\rho$
induces a $*$-isomorphism $\rho^* : A_S(\G_E) \to A_S(\G_F)$ satisfying $\rho^*(f) = f
\circ \rho$. In particular, $\rho$ carries $D_{\G_E} \subseteq A_S(\G_E)$ to $D_{\G_F}
\subseteq A_S(\G_F)$. So $\alpha^{-1}_F \circ \rho^* \circ \alpha_E$ is a $*$-isomorphism
$L_S(E) \to L_S(F)$ as required in~(\ref{it:cor:graphs3}). Similarly the isomorphism
$\rho$ induces a C*-algebra isomorphism $\rho^* : C^*(\G_E) \to C^*(\G_F)$ satisfying
$\rho^*(f) = f\circ\rho$ for $f \in C_c(\G_E)$. In particular $\rho^*(C_0(\G_E^{(0)})) =
C_0(\G_F^{(0)})$. It is standard that there is an isomorphism $\phi_E : C^*(E) \to
C^*(\G_E)$ that carries $\clsp\{s_\mu s^*_\mu : \mu \in E^*\}$ to $C_0(\G_E^{(0)})$, and
similarly a diagonal-preserving isomorphism $\phi_F : C^*(F) \to C^*(\G_F)$. So $\psi :=
\phi^{-1}_F \circ \rho^* \circ \phi_E$ is the isomorphism in~(\ref{it:cor:graphs5}).

Now we prove (\ref{it:cor:graphs3})~implies~(\ref{it:cor:graphs1}). Suppose
that~(\ref{it:cor:graphs3}) holds. Taking $S = \mathbb{F}_2$ (the field of two elements),
for example, trivially gives~(\ref{it:cor:graphs1}).

For (\ref{it:cor:graphs5})~implies~(\ref{it:cor:graphs1}), suppose
that~(\ref{it:cor:graphs5}) holds. With $\phi_E$ and $\phi_F$ as above, the map $\phi_F
\circ \psi \circ \phi^{-1}_E$ is an isomorphism $C^*(\G_E) \to C^*(\G_F)$ that carries
the Cartan subalgebra $C_0(\G_E^{(0)})$ to the Cartan subalgebra $C_0(\G_F^{(0)})$. So
\cite[Proposition~4.13]{Ren2008} implies that there is an isomorphism $\rho : \G_F \cong
\G_E$ as in the first paragraph. This induces an isomorphism $\rho^* : A_R(\G_E) \to
A_R(\G_F)$ that takes $D_{\G_E}$ to $D_{\G_F}$. So $\alpha^{-1}_F \circ \rho^* \circ
\alpha_E$ is the desired isomorphism of Leavitt path algebras.

It remains to prove (\ref{it:cor:graphs4})\;$\iff$\;(\ref{it:cor:graphs5}). The
implication (\ref{it:cor:graphs4})~implies~(\ref{it:cor:graphs5}) is trivial; and
(\ref{it:cor:graphs5})~implies~(\ref{it:cor:graphs4}) by the same argument as we used for
(\ref{it:cor:graphs2})~implies~(\ref{it:cor:graphs1}) because Renault's theorems prove
that $\clsp\{s_\eta s^*_\eta : \eta \in F^*\}$ is a maximal abelian subalgebra of
$C^*(F)$.
\end{proof}

\begin{remark}\label{rmk:BCaH}
We learned of the paper \cite{BCaH} in the later stages of the preparation of this
manuscript. Our Corollary~\ref{cor:graphs} is related to the main theorem
\cite[Theorem~5.3]{BCaH}, though neither strictly generalises the other. There are two
differences between the two results:
\begin{itemize}
\item Theorem~5.3 of \cite{BCaH} applies to row-finite graphs $E$ and $F$ with no
    sinks, whereas our result applies to arbitrary graphs $E$ and $F$ in which
    every cycle has an exit.
\item The hypotheses of \cite[Theorem~5.3]{BCaH} demand that the isomorphism $\pi :
    L_R(E) \to L_R(F)$ should be a $^*$-isomorphism, and that it should restrict to
    an isomorphism $\pi : D_E \to D_F$; whereas Corollary~\ref{cor:graphs} requires
    only a ring isomorphism $L_R(E) \to L_R(F)$ that carries $D_E$ into the commutant
    of $D_F$.
\end{itemize}
\end{remark}

We use our results to obtain an improvement of \cite[Theorem 3.6]{JS}. For its statement, we need some standard
graph-theoretical definitions, as follows.

\begin{definition}
A graph $E$ is said to be:
\begin{enumerate}
 \item {\it strongly connected} if there is a path between any two vertices.
 \item {\it essential} if it has no sinks or sources, and
 \item {\it trivial} if it is a single cycle with no other edges or vertices.
\end{enumerate}
 \end{definition}

\begin{corollary} \label{cor:JS-improvement}
Let $E,F$ be finite, essential, non-trivial, strongly connected graphs, and let $R$ be
any commutative integral domain with~1. If there is an isomorphism $\phi\colon L_R(E) \to
L_R(F)$ such that $\phi (\mathcal D (L_R(E)))\subseteq \mathcal D (L_R(F))$, then
$$\mathrm{sgn}(\mathrm{det}(I- A_E)) = \mathrm{sgn}(\mathrm{det}(I - A_F)).$$
\end{corollary}

\begin{proof}
It is straightforward to check that the conditions on $E$ and $F$ imply that both graphs have the property that every cycle has an exit.
By Corollary \ref{cor:graphs}, we obtain a $C^*$-algebra isomorphism $\overline{\phi}\colon C^*(E)\to C^*(F)$ such that
$\overline{\phi}(\mathcal D (C^*(E))) = \mathcal D (C^*(F))$. It follows from \cite[Theorem 3.3]{JS} (cf. \cite[Theorem 3.6]{MM}) that
$\mathrm{sgn}(\mathrm{det}(I-A_E)) = \mathrm{sgn}(\mathrm{det}(I - A_F)).$
\end{proof}

This result can be applied to give a partial answer to one of the most intriguing open questions in the theory of Leavitt path algebras; namely, whether, for a
commutative coefficient ring $R$ with~1, the algebras $L_{2,R}$ and $L_{2-, R}$ are isomorphic. Johansen and S\o rensen have recently shown that
there is no $*$-isomorphism between $L_{2,\mathbb Z}$ and $L_{2-,\mathbb Z}$ (\cite{JS}). Recall from e.g. \cite{JS} that $L_{2,R}$
denotes the classical Leavitt algebra
of type $(1,2)$ with coefficients in $R$. It is the Leavitt path $R$-algebra of the graph $E_2$ with one vertex and two arrows.
The algebra $L_{2-,R}$ is the Leavitt path $R$-algebra associated to a graph $E_{2-}$ depicted in the introduction to \cite{JS}.
Over any regular supercoherent coefficient
ring $R$, both algebras $L_{2,R}$ and $L_{2-,R}$ have trivial algebraic $K$-theory (\cite{ABC}). However, they are distinguished by the numbers appearing in
Corollary \ref{cor:JS-improvement}.

\begin{corollary} \label{distinguishingLs}
Let $R$ be a commutative integral domain with~1. Then there is no isomorphism $\phi
\colon L_{2,R}\to L_{2-,R}$ such that  $\phi(\mathcal D (L_{2,R}))\subseteq \mathcal D
(L_{2-, R})$ or $\phi^{-1} (\mathcal D (L_{2-, R}))\subseteq \mathcal D (L_{2,R})$.
\end{corollary}
\begin{proof}
Assume  there is an isomorphism $\phi \colon L_{2,R}\to L_{2-,R}$ such that
$\phi(\mathcal D (L_{2,R}))\subseteq \mathcal D (L_{2-, R})$ or $\phi^{-1} (\mathcal D
(L_{2-, R}))\subseteq \mathcal D (L_{2,R})$. The graphs $E_2$ and $E_{2-}$ are finite,
essential, non-trivial and strongly connected. Therefore, it follows from Corollary
\ref{cor:JS-improvement} that
\[
\mathrm{sgn}(\mathrm{det}(I- A_{E_2}))
    = \mathrm{sign}(\mathrm{det}(I - A_{E_{2-}})).
\]
However $\mathrm{det}(I- A_{E_2})= -1$ and $\mathrm{det}(I - A_{E_{2-}})=+1$, so we
obtain a contradiction.
\end{proof}

 \subsection{Graded ring-isomorphisms of Kumjian--Pask algebras}

In this section, we emphasise what extra information we obtain by keeping track of the
graded structure in Section~\ref{sec:main}. Recall that for every $k$-graph $\Lambda$,
the associated $k$-graph groupoid $\G_\Lambda$ (see \cite{KP2000} or \cite{FMY}) is
$\mathbb{Z}^k$-graded, and $c^{-1}(0)$ is a principal groupoid. So our main theorem
yields the following:

\begin{corollary}
Suppose that $\Lambda$ and $\Gamma$ are $k$-graphs and that $R$ is a commutative integral
domain with~1. There is a graded ring-isomorphism $\phi : \operatorname{KP}_R(\Lambda)
\cong \operatorname{KP}_R(\Gamma)$ such that $\phi(s_\mu s_{\mu^*})s_\eta s_{\eta^*} =
s_\eta s_{\eta^*}\phi(s_\mu s_{\mu^*})$ for all $\mu \in \Lambda$ and $\eta \in \Gamma$
if and only if the groupoids $\G_\Lambda$ and $\G_\Gamma$ are isomorphic, in which case
there is a diagonal preserving isomorphism $\operatorname{KP}_S(\Lambda) \cong
\operatorname{KP}_S(\Gamma)$ for every ring $S$, and there is a diagonal-preserving
isomorphism $C^*(\Lambda) \cong C^*(\Gamma)$.
\end{corollary}
\begin{proof}
The argument is essentially the same as the corresponding implications in
Corollary~\ref{cor:graphs}, except that we apply Theorem~\ref{thm:main} instead of
Corollary~\ref{cor:ungraded gpds}.
\end{proof}

\end{document}